\documentclass[a4paper,11pt]{article}

\usepackage{iftex}
\ifXeTeX
\usepackage{fontspec}
\else
\usepackage[T1]{fontenc}
\fi
\usepackage[english]{babel}
\usepackage{amssymb,amsfonts}
\usepackage{mathtools}
\usepackage{esint}
\usepackage{amsmath,amsthm,hyperref}
\usepackage{graphicx}
\usepackage{fancyhdr}
\usepackage{upgreek}
\usepackage{color}
\usepackage{enumitem}
\setlist[enumerate,1]{label=(\arabic*)}
\setlist[enumerate,2]{label=(\alph)}
\usepackage[toc,page]{appendix}
\usepackage{thmtools,thm-restate}
\usepackage{xfrac}

\usepackage{tabularx}
\usepackage{verbatim}

\definecolor{carminered}{rgb}{1.0, 0.0, 0.22}
\definecolor{islamicgreen}{rgb}{0.0, 0.56, 0.0}
\definecolor{blue(ryb)}{rgb}{0.01, 0.28, 1.0}
\definecolor{green(ncs)}{rgb}{0.0, 0.62, 0.42}

\newcommand{\calH}{\mathcal{H}}

\newtheorem{proposition}{Proposition}[section]
\newtheorem{lemma}[proposition]{Lemma}
\newtheorem{remark}[proposition]{Remark}
\newtheorem*{remark*}{Remark}
\newtheorem{theorem}[proposition]{Theorem}

\newtheorem*{lemma*}{Lemma}
\newtheorem{assumption}[proposition]{Assumption}

\let\oldappendices\appendices
\renewcommand{\appendices}{\color{green(ncs)}\oldappendices}

\newcommand{\ignore}[1]{}

\newcommand{\sphere}{\mathcal{S}^2}
\newcommand{\eps}{\varepsilon}

\newcommand{\R}{\mathbb{R}}
\newcommand{\N}{\mathbb{N}}

\DeclareMathOperator{\loc}{loc}

\renewcommand{\sp}{\text{\textup{axs}}}
\numberwithin{equation}{section}

\usepackage{hyperref}
\hypersetup{
	pdfauthor={},
	pdftitle={},
	pdfsubject={},
	urlcolor=blue,
}

\newcommand{\interior}[1]{%
	{\kern0pt#1}^{\mathrm{o}}%
}

\usepackage{filecontents}

\begin{document}

	\title{Interface behavior for the solutions of a mass conserving free boundary problem modeling cell polarization}
	\author{Anna~Logioti\thanks{Institute of Analysis, Dynamics and Modeling, University of Stuttgart} \quad Barbara~Niethammer\thanks{Institute for Applied Mathematics,
			University of Bonn} \quad  Matthias~R\"oger\thanks{Department of Mathematics,
			TU Dortmund University}  \quad Juan~J.~L.~Vel\'azquez\footnotemark[2]}
	\maketitle
	
	\begin{abstract}
		We consider a parabolic non-local free boundary problem that has been derived as a limit of a bulk-surface reaction-diffusion system which models cell polarization. In  previous papers \cite{LNRV21,LNRV23} we have established well-posedness of this problem and derived conditions on the initial data that imply continuity of the free boundary as $t \to 0$.  
		In this paper we extend the qualitative study of the free boundary by considering axisymmetric data. Under additional monotonicity assumptions on the data we prove global continuity of the free boundary. On the other hand,  if  the initial data violate a "no-fattening" condition we show that the free boundary can oscillate as $t \to 0$.
		\medskip
		
		
	\end{abstract}
	
	\tableofcontents
	
	\section{Introduction}
	
	\subsection{Background} 
	
	In this paper we discuss qualitative properties of a nonlocal free boundary problem that has been derived from a bulk-surface reaction diffusion system \cite{NRV20,LNRV23} as a model for cell polarization. 
	
	In the following let $T>0$ be an arbitrary time and  $\Gamma \subset \R^3$ be a smooth compact surface without boundary that represents the membrane of a cell. The nonnegative function $u \colon \Gamma \times [0,T) \to [0,\infty)$ denotes the density of a certain protein on $\Gamma$ while $g \colon \Gamma \times [0,T) \to (0,1]$ is a given function representing a chemical signal. Then, if $H \colon \R \to \{0,1\}$ denotes the Heaviside function, and $u_0 \colon \Gamma \to [0,\infty)$ initial data, the free boundary problem can be stated as follows (see \cite[Lemma 2.3]{LNRV23}) 
	
	\begin{align}
		\partial_t u -\Delta u &=-\Big(1-\frac{g}{\lambda(t)} \Big )H(u) &\mbox{ a.e. on }\;\Gamma_T:=\Gamma \times (0,T)\;,
		\label{2eq:ob2}\\
		\lambda(t) &:= \fint_{\{u(\cdot,t)>0\}}g\,dS\;, & \mbox{ for a.a. } t \in (0,T)\,, \label{1eq:ob4}\\
		g &\leq\lambda(t)  &\mbox{ a.e. in  }\{u(\cdot,t)=0\}\;,	\label{2eq:ob3}\\
		u &\geq 0  &\mbox{ a.e. on }\;\Gamma \times (0,T)\;\label{2eq:ob1}\\
		u(\cdot,0) & = u_0 & \mbox{ a.e. on } \Gamma \;.  \label{1eq:ob5}
	\end{align}
	where $\Delta_\Gamma$ denotes the Laplace-Beltrami operator on $\Gamma$.
	The function $\lambda$ in \eqref{1eq:ob4} can be understood as a Lagrange multiplier that guarantees  mass conservation, i.e.
	\begin{equation}
		\int_\Gamma u(\cdot,t)\,dS = \int_{\Gamma} u_0\,dS \qquad \mbox{ for all } t\in [0,T]\,.  \label{1eq:massconservation}
	\end{equation}
	One key property of the parabolic obstacle-type problem \eqref{2eq:ob2}-\eqref{1eq:ob5} is the particular nonlocality in form of a dependence on the {support} of the solution.

	For the system \eqref{2eq:ob2}-\eqref{1eq:ob5} we say that we have a polarized state if both, the zero set of  $u(\cdot,t)$ and its complement on $\Gamma$ have nonzero measure. In our previous work \cite{NRV20}  we proved existence and uniqueness of steady states for a  given mass and in addition we characterized the critical mass below which polarization occurs. Well-posedness for the full parabolic problem as well
	as global stability of steady states has been established in \cite{LNRV21}. 
	
	The purpose of this paper is to continue a qualitative study of the parabolic free boundary problem that we have started in \cite{LNRV23}. More precisely, in \cite{LNRV23} we have derived the two following conditions on the initial data which ensure that the positivity set $\{u(\cdot,t)>0\}$ changes continuously as $t\to 0^+$.
	
	First, we require that 
	\begin{equation}\label{eq:nondeg1}
		g-\lambda(0)\leq -\theta<0  \qquad\mbox{ in } \{u_0=0\}
	\end{equation}
	for some fixed $\theta>0$, and, second, we assume  a `non-fattening' of the boundary of the support of the initial data, i.e.
	\begin{equation}\label{eq:nondeg2}
		\calH^2\big(\partial \{u_0>0\}\big) =0\;.
	\end{equation}
	If \eqref{eq:nondeg2} holds we show in \cite{LNRV23} that condition \eqref{eq:nondeg1} is necessary and sufficient to obtain the continuity of solutions at $t=0$ from the right.

	In general, however, in \cite{LNRV23} it remained open whether \eqref{eq:nondeg2} is necessary for the solution to be right-continuous at $t=0$ and whether the solution remains continuous for positive times.
	
	In this paper, we first prove in Section \ref{sec:continuity}  a global continuity result for the support of the solution $u(\cdot,t)$ under additional symmetry and monotonicity assumptions on the data. Second, we prove a result that indicates that  if  condition \eqref{eq:nondeg2} fails, then the function $\lambda$ is not necessarily right-continuous at $t=0$ even if \eqref{eq:nondeg1} holds.
	More precisely, we provide an example of initial data $u_0$ for which \eqref{eq:nondeg2} is not valid and such that the support of $u(\cdot,t)$ and $\lambda(t)$ behave oscillatory with $t\downarrow 0$. We prove this result rigorously for the classical parabolic obstacle problem in Section \ref{sec:obsc-obs}, which is of interest in its own, and then  extend this result to a slightly simplified nonlocal free boundary problem in Section \ref{sec:nonlocal-obs-R}.
	
	\medskip
	In the rest of this paper we restrict ourselves to the specific case of spherical geometry, $\Gamma=\mathcal S^2$ and to axisymmetric data and axisymmetric solutions.
	Therefore the  problem is reduced to a one-dimensional spatial dependence.
	First, we collect some results from previous works that will play a crucial role in the current analysis.
	In Section \ref{sec:continuity} we prove the global in time continuity of the positivity set $\{u(\cdot,t)>0\}$.
	
	The following sections provide examples of an oscillatory behavior of solutions if the second non-degeneracy condition is violated.
	Section \ref{sec:obsc-obs} considers first the classical parabolic obstacle problem on the real line.
	Finally, in Section \ref{sec:nonlocal-obs-R} we present a corresponding oscillation result for a nonlocal analogue of the system \eqref{2eq:ob2}-\eqref{1eq:ob5} on the real line.

	\subsection{Preliminaries}
	\label{sec:prelim}
	We have established in \cite{LNRV21} that problem \eqref{2eq:ob2}-\eqref{1eq:ob5} admits a unique nonnegative global solution.
	More precisely, we prove that for any $T>0$ and any nonnegative $u_0 \in L^2(\Gamma)$, there exists a unique solution $u \in L^2(0,T;H^1(\Gamma)) \cap H^1(0,T;H^1(\Gamma)^*)$ and further we show that $u\in L^p\big(\delta,T;W^{2,p}(\Gamma)\big)\cap W^{1,p}\big(\delta,T;L^p(\Gamma)\big)$ for any $\delta>0,\,1\leq p<\infty$. 
	By classical embedding arguments it follows that $u\in C^{1+\beta,\frac{1+\beta}{2}}(\Gamma\times [\delta,T])$ for all $0<\beta<1$.
	
	In the case that $u_0\in H^2(\Gamma)$ we even have $u\in L^p\big(0,T;W^{2,p}(\Gamma)\big)\cap W^{1,p}\big(0,T;L^p(\Gamma)\big)$ for any $1\leq p<\infty$, and $u\in C^{1+\beta,\frac{1+\beta}{2}}(\Gamma\times [0,T])$ for all $0<\beta<1$, see for example \cite[Lemma II.3.3]{LaSU68}.
	
	Finally we remark, that by the uniform convergence to a unique stationary state and the estimates provided in \cite{LNRV21} we even have that $u$ is uniformly bounded in $\Gamma\times [0,\infty)$.
	

	From now on we will restrict ourselves to the spherical case $\Gamma=\sphere$ and to axisymmetry with respect to the first coordinate axis. 
	
	\begin{remark}[Axisymmetric data]
		\label{2rem:axis}
		We consider functions $U$ on the sphere $\sphere\subset\R^3$ given by a function $u$ on $[-1,1]$ by
		\begin{equation*}
			U(x_1,x_2,x_3) = u_{\sp}(x_1,x_2,x_3) := u(x_1)\,.
		\end{equation*}
		For functions that depend on a scalar space variable we denote the derivative just by a prime, in particular $u'(x_1)=\frac{d}{dx_1}u$ et cetera.
		
		The Laplace-Beltrami operator for axisymmetric functions $U$ as above can then be written as
		\begin{equation}\label{LB-axi}
			\Delta_{\sphere} U(x)=\big((1-x_1^2) u'(x_1)\big)'\,.
		\end{equation}
		Moreover, $u\in L^1(-1,1)$ if and only if $U\in L^1(\sphere)$, and it holds
		\begin{equation}\label{mass-axi}
			\int_{\sphere} u\,dS = 2 \pi \int_{-1}^1 u(r)\,dr\,.
		\end{equation}
		
		We denote by $C^1_{\sp}([-1,1])$ and $W^{k,p}_{\sp}(-1,1)$ the space of functions $u$ such that $u_{\sp}\in C^1(\Gamma)$ and $u\in W^{k,p}(\sphere)$, respectively.
		We remark that $u\in C^1(\sphere)$ if and only if $\tilde u\in C^1(-1,1)$ with 
		\begin{equation}
			\lim_{r\downarrow 0}r u'(\pm \sqrt{1-r^2})=0\,,
			\label{2eq:D1-cond}
		\end{equation}
		which in particular reflects that $\nabla u(0,0,\pm 1)=0$ holds, 
		and that $u\in W^{1,p}_{\sp}(-1,1)$ if and only if $u\in L^q(-1,1)\cap W^{1,q}_{\loc}(-1,1)$ with $x\mapsto \sqrt{1-x^2} u'(x)\in L^q((-1,1))$.
		
		The parabolic Hölder and Sobolev spaces $C^{2k+\alpha,k+\frac{\alpha}{2}}_{\sp}([-1,1]\times [0,T])$ and $W^{2,1}_{p,\sp}((-1,1)\times (0,T))$, respectively, are defined analogously.
	\end{remark}
	
	\begin{lemma}
		Consider $\Gamma=\sphere$, $1\leq p<\infty$ and an axisymmetric function $U\in W^{2,1}_p(\sphere\times (0,T))$ represented by some $u\in W^{2,1}_{p,\sp}((-1,1)\times (0,T))$ as $U=u_{\sp}$.
		
		Then $U$ solves \eqref{2eq:ob2} (with $u$ replaced by $U$) if and only if $u$ solves 
		\begin{equation}
			\partial_t u-\big((1-x^2)u'\big)'=-\Big(1-\frac{g}{\lambda(t)} \Big ) H(u)\; \qquad \text{in }\; (-1,1) \times (0,T]\,.
			\label{3eq:ob2}
		\end{equation}
		
		Moreover, for given axisymmetric data $U_0=u_{0,\sp}\in H^2(\sphere)$ and $G=g_{\sp}\in C^0(\sphere)$ there is a one-to-one correspondence between solutions $U$ of \eqref{2eq:ob2}-\eqref{1eq:ob5} (with $u,g,u_0$ replaced by $U,G,U_0$) and solution $u$ to the nonlocal initial boundary problem given by \eqref{3eq:ob2} and
		\begin{align}
			\lambda(t) &:= \fint_{\{u(\cdot,t)>0\}}g(r)\,dr\;, & \text{ for a.a. } t \in (0,T)\,, \label{3eq:ob4}\\
			g &\leq\lambda(t)  &\text{ a.e. in  }\{u(\cdot,t)=0\}\;,	\label{3eq:ob3}\\
			u &\geq 0  &\text{ a.e. on }\;(-1,1) \times (0,T)\;\label{3eq:ob1}\\
			u(\cdot,0) & = u_0 & \text{ a.e. on } (-1,1)\;.  \label{3eq:ob5}
		\end{align}
	\end{lemma}
	
	\begin{proof}
		This follows by the formulas stated in Remark \ref{2rem:axis} and by the existence and uniqueness result of solutions to \eqref{2eq:ob2}-\eqref{1eq:ob5} proved in \cite{LNRV21}.
	\end{proof}

	\section{Global continuity result for axisymmetric solutions} 
	\label{sec:continuity}
	
	From now on we will only deal with the axisymmetric case.
	
	Our aim in this section is to provide global continuity results under certain  assumptions on the initial data $u_0$ and the external stimulus $g$.
	\begin{assumption}\label{ass:main}
		We assume that 
		\begin{equation}\label{initialdata}
			u_0\in C^2_{\sp}([-1,1]) \quad\text{ with }
			\; u_0 \geq 0 \quad\text{ and }
			\quad |\{u_0>0\}|>0
		\end{equation}
		and for some $\gamma \in [-1,1)$
		\begin{equation}\label{data1d}
			\{ u_0>0\}=(\gamma,1] \quad\text{ and }\quad u_0'> 0 \quad\text{ a.e. in  }(\gamma,1)\;.
		\end{equation}
		Furthermore, we assume that 
		\begin{equation}\label{gassumptions}
			g \in C^2_{\sp}([-1,1])  \qquad \mbox{ with } \quad 0< g_0\leq g \leq g_1 <1 \;\mbox{ on } \mathcal S^2\,\;
		\end{equation}
		for some constants $0<g_0<g_1<1$ and that for some $\kappa>0$ we have 
		\begin{equation}\label{g1d}
			g' \geq \kappa>0 \quad\text{ in  }[-1,1]\,.
		\end{equation}
	\end{assumption}
	
	For the following we define the boundary of the positivity set of $u$ via
	\begin{equation}\label{pdef}
		p(t):=\inf \{ x\,|\, u(x,t)>0\}\,.
	\end{equation}
	Indeed, we will see in Lemma \ref{L.1} that if $u_0$ is increasing, then so is $u(\cdot,t)$ for any $t>0$ and $[p(t),1]$ is indeed  the support of $u(\cdot,t)$.

	\begin{theorem}\label{main_thm}
		Suppose that Assumption \ref{ass:main} holds. 
		Moreover, let $u\in W^{2,1}_{p,\sp}([-1,1]\times [0,\infty)$ be a solution to \eqref{3eq:ob2}-\eqref{3eq:ob5}. 
		Then $p\colon [0,\infty) \to [-1,1]$ is continuous.
	\end{theorem}
	
	The proof will be given at the end of this section.
	We start with some auxiliary results.
	
	\subsection{Monotonicity and non-degeneracy}
	
	First we prove that the monotonicity property assumed for the initial data propagates to positive times.
	
	\begin{lemma}\label{L.1}
		Under the assumptions of Theorem \ref{main_thm}  
		$u'(\cdot,t) \geq 0$ holds for all $t\geq 0$.
	\end{lemma}
	
	\begin{proof}
		We multiply \eqref{3eq:ob2} by $\big((1-x^2)u'(x,t)\big)'$ and integrate over $\{u'(\cdot,t)>0\}\subset \{u(\cdot,t)>0\}$.
		We use
		\begin{equation*}
			\int_{\{u'(\cdot,t)>0\}}\partial_tu(x,t)\big((1-x^2)u'(x,t)\big)'\,dx
			=-\frac{d}{dt}\int_{\{u'(\cdot,t)>0\}}\frac{1}{2}\big((1-x^2)u'(x,t)\big)^2\,dx\,,
		\end{equation*}
		in a weak sense, and
		\begin{equation*}
			-\int_{\{u'(\cdot,t)>0\}}\Big(1-\frac{g(x)}{\lambda(t)}\Big)\big((1-x^2)u'(x,t)\big)'\,dx
			=-\frac{1}{\lambda(t)}\int_{\{u'(\cdot,t)>0\}}g'(x)(1-x^2)u'(x,t)\,dx \geq 0.
		\end{equation*}
		Therefore we obtain
		\begin{equation*}
			\frac{d}{dt}\int_{\{u'(\cdot,t)>0\}}\frac{1}{2}\big((1-x^2)u'(x,t)\big)^2\,dx \leq 0\,,
		\end{equation*}
		which implies by $u_0'\geq 0$ that $u'(\cdot,t)\geq 0$ almost everywhere in $(-1,1)$ for almost all $t\in (0,\infty)$.
		By continuity of $u'$, see the remarks at the beginning of Section \ref{sec:prelim} the claim follows.
	\end{proof}
	
	Next, we show that the support of $u$ can not be arbitrarily small.
	For that we recall that for any $T>0$ we have a bound $\|u\|_{\infty}:=\|u\|_{L^{\infty}((-1,1)\times [0,T])} \leq C_T$.
	
	\begin{lemma}\label{pless1}
		Under the assumptions of Theorem \ref{main_thm}
		it holds
		\begin{equation}\label{Upper}
			p(t) \leq 1-\frac{m}{2\pi \|u\|_{\infty}} \; \qquad \text{for all }\;t\in [0,T]\,.
		\end{equation}
		Furthermore we have
		\begin{equation}\label{suffC2b}
			g(x) \leq \lambda(t)-\frac{\kappa m }{4\pi \|u\|_{\infty}} \qquad \mbox{ for all } x \in (-1,p(t)) \mbox{ and } t\in [0,T]\,.
		\end{equation}
	\end{lemma}
	\begin{proof}
		We observe that due to \eqref{1eq:massconservation}, \eqref{mass-axi} estimate \eqref{Upper} follows from
		\begin{align*}
			m=2\pi \int_{p(t)}^1 u(x,t)\;dx\leq 2\pi {\|u\|}_{\infty} (1{-}p(t))\;.
		\end{align*}
		By  Taylor's Theorem and \eqref{g1d}, we obtain that
		\begin{align*}
			\lambda(t)&=\frac{1}{1{-}p(t)} \int_{p(t)}^1 g\;dx \geq \frac{1}{1{-}p(t)} \int_{p(t)}^1 g(p(t))+\kappa(x-p(t))\;dx\\ &\geq g(p(t))+\frac{\kappa}{2}(1{-}p(t))
			\geq g(p(t)) + \frac{\kappa m}{4\pi \|u\|_{\infty}}\,.
		\end{align*}
		Due to the monotonicity of $g$ we deduce \eqref{suffC2b}.
	\end{proof}

	Next we prove that $p(t)$ is sufficiently separated from the value $s(t)$ where $g(s(t))=\lambda(t)$.
	We therefore obtain a uniform non-degeneracy property of the solutions, compare assumption (1.9) in \cite{LNRV23} and the discussion in that paper.
	
	\begin{lemma}\label{L.2}
		Assume that \eqref{data1d}, \eqref{g1d} are valid.
		Then, there exists a unique $s(t)\geq p(t)$ with
		\begin{equation}
			g(s(t))=\lambda(t)\; \label{critical}
		\end{equation}
		and it holds
		\begin{equation}\label{interface}
			p(t) \leq s(t)- c_0
			\qquad \mbox{ for all } t\in [0,T]
		\end{equation}
		for some $c_0=c_0(g,m,\|u\|_{\infty})>0$.
	\end{lemma}
	
	\begin{proof}
		We notice due to \eqref{1eq:massconservation}, \eqref{gassumptions} and \eqref{g1d} that $g(-1)< \lambda(t) < g(1)$ for all $t\in [0,T]$.
		Thus, there exists $s(t) \in (-1,1)$ such that \eqref{critical} holds true.	Furthermore, due to \eqref{g1d}, the function $g^{-1}:[g(-1),g(1)] \to [-1,1]$ is well defined.
		Then, recalling that $g(p(t))\leq \lambda(t)$, it follows that $p(t)\leq s(t)$.
		Inequality \eqref{interface} then follows from \eqref{suffC2b} and \eqref{g1d}.
	\end{proof}

	\medskip
	Finally, we formulate a non-degeneracy lemma, analogous to \cite[Lemma 3.3]{LNRV23}. 
	
	\begin{lemma}\label{L.degeneracy}
		Let $0\leq t_1\leq t_2$, $x_0\in (-1,1]$ and $0<\rho\leq 1$ are given such that $x_0+2\rho \leq p(t_1)$ and such that for $u\in W^{2,1}_{p,\sp}((-1,1)\times (t_1,t_2))$
		\begin{equation}
			\partial_t u-\big((1-x^2)u'\big)'\leq -\theta H(u)\quad
			\text{ in } (-1,x_0+2\rho) \times (t_1,t_2)\,,
			\label{3eq:LemDeg}
		\end{equation}
		Moreover, assume that
		\begin{equation}
			u \leq \theta \rho^2 \quad\text{ on } (-1,x_0+2\rho) \times (t_1,t_2)\,.
			\label{3eq:AssNonDeg}
		\end{equation}
		Then $u=0$ holds in $(-1,x_0+\rho)\times (t_1,t_2)$.
		This in particular implies that $p(t) \geq x_0 + \rho$ in $[t_1,t_2]$.
	\end{lemma}
	
	\begin{proof}
		Without loss of generality we may assume $t_1=0$ and set $T=t_2$, $p_0=p(0)$.
		
		Consider the function $U=u_{\sp}\in W^{2,1}_p(\sphere\times (0,T))$.
		Then $U$ solves \eqref{2eq:ob1} in $\big(\sphere\cap \{x\cdot\vec{e}_1<x_0+2\rho\}\big)\times (0,T)$ and satisfies $U(\cdot,0)=0$ in $\big(\sphere\cap \{x\cdot\vec{e}_1<x_0+2\rho\}\big)$.
		
		Assume by contradiction that $u(y_1,t_1)>0$ for some $-1\leq y_1<x_0+\rho$ and $0<t_1<T$.
		Let $y:= (y_1,0,\sqrt{1-y_1^2})$.
		Consider the comparison function 
		\begin{equation*}
			Q:B_{\sphere}(y,\rho)\times (0,t_1)\to\R,\quad
			Q(x,t)= \theta\big(|x-y|^2 + (t_1-t)\big)\,,
		\end{equation*}
		where $B_{\sphere}(y,\rho)$ refers to the respective ball in $\sphere$ with respect to the distance in $\R^3$.
		
		Then
		\begin{equation*}
			\big(\partial_tQ-\Delta Q\big)(x,t) = \theta\big(-1 -4  -2\vec{H}_{\sphere}(x)\cdot(x-y)\big)
			= \theta\big(-1-4+2-2x\cdot y \big)\leq -\theta\,.
		\end{equation*}
		We use a comparison principle on $\big(B_{\sphere}(y,\rho)\times (0,t_1)\big)\cap \{U>0\}$.
		
		We first observe that $B_{\sphere}(y,\rho)\subset  \{x\cdot\vec{e}_1<x_0+2\rho\}$ and that $Q\geq U$ holds on $\partial \{U>0\}$ and in $B_{\sphere}(y,\rho)\times\{0\}$.
		Furthermore on $\{x\in\sphere\,:\, |x-y|=\rho\}$ we have $Q\geq \theta\rho^2\geq U$ by \eqref{3eq:AssNonDeg}.
		
		It follows that $U(y,t_1)\leq Q(y,t_1)=0$, a contradiction to our assumption.
	\end{proof}

	\subsection{Proof of Theorem \ref{main_thm}}
	We define
	\begin{equation}\label{deltanulldef}
		\delta_0:=\frac 1 8 \min \Big (c_0,\frac{m}{2\pi \|u\|_{\infty}},\frac{\kappa m}{4 \pi \|u\|_{\infty}}\Big)
	\end{equation}
	with $c_0$ as in \eqref{interface}.
	
	{\bf Step 1:} 
	We prove the uniform lower-semicontinuity from the right, more precisely:
	For any $\tilde t\in [0,T)$ and any $\delta>0$ there exists $\omega(\delta)$ only depending on $g,m,\|u\|_{C^{\alpha,\alpha/2}([-1,1]\times[0,T])}$ such that
	\begin{equation}\label{lowerbound}
		p(\tilde t) - \delta \leq p(t) \qquad \mbox{ for all } t \in [\tilde t, \tilde t + \omega(\delta)]\,.
	\end{equation}
	Without loss of generality we can assume $\tilde t=0$ in the following.
	
	We fix any $\delta \in (0,\delta_0]$ and suppose that $p(0)\geq -1+\delta$ since otherwise there is nothing to show. 
	
	By the assumption \eqref{initialdata} on the initial data and the remarks at the beginning of Section \ref{sec:prelim} for any fixed $\alpha>0$ we can choose $\omega>0$, $\omega=\omega(\delta)$ such that
	\begin{equation}
		\|u\|_{C^{\alpha,\alpha/2}([-1,1]\times [0,T])} \omega^{\frac{\alpha}{2}}
		\leq \frac{\kappa}{g_{\max}}\frac{c_0}{8}4\delta^2\,.
		\label{eq:t0def}
	\end{equation}
	Next we define
	\begin{equation*}
		s^*(\omega) := \inf\big\{s(\tau)\,:\, 0\leq \tau\leq \omega\big\}\,.
	\end{equation*}
	Due to \eqref{interface} we always have $s^*(\omega)\geq -1+c_0$.
	
	We set
	\begin{equation}
		b:=\min\Big\{p(0),s^*(\omega)-\frac{c_0}{8}\Big\}\,.
		\label{eq:defb}
	\end{equation}
	For any $0\leq t\leq \omega$ and any $-1\leq x\leq b$ we have $b\leq s(t)-\frac{c_0}{8}$ and
	\begin{equation*}
		g(x) -\lambda(t) \leq g(b) -g(s(t)) \leq -\kappa \frac{c_0}{8}\,,
	\end{equation*}
	hence
	\begin{equation}
		1-\frac{g(x)}{\lambda(t)} \geq  \frac{\kappa}{\lambda(t)}\frac{c_0}{8} 
		\geq \frac{\kappa}{g_{\max}}\frac{c_0}{8}\,.
	\end{equation}
	We deduce that
	\begin{align}
		&\partial_t u - \big((1-x^2)u'\big)' \leq -\frac{\kappa}{g_{\max}}\frac{c_0}{8}H(u)
		\quad&&\text{ on }(-1,b)\times (0,\omega)\,,\\
		&u(\cdot,0 )=0\quad&&\text{ in }(-1,b)\,,
	\end{align}
	where we have used $b\leq p(0)$.
	
	Due to \eqref{eq:t0def} and Lemma \ref{L.degeneracy} we deduce that
	\begin{equation}
		u(x,t)=0 \quad\text{ for all }-1\leq x\leq b-\delta\,,\, 0\leq t\leq \omega\,,
	\end{equation}
	which implies
	\begin{equation}
		p(t)\geq b-\delta\quad\text{ for all } 0\leq t\leq \omega\,.
		\label{eq:lbb}
	\end{equation}
	Assume that $b=s^*(\omega)-\frac{c_0}{8}$ in \eqref{eq:defb}.
	However, \eqref{interface} and \eqref{eq:lbb} then yield $s^*(\omega)\geq b-\delta+c_0\geq b+\frac{7c_0}{8}$, a contradiction.
	
	Therefore $b=p(0)$ and by \eqref{eq:lbb} the claim is proved.
	
	\medskip
	\newpage
	{\bf Step 2:} 
	We now prove the uniform lower-semicontinuity from the right, more precisely: for any $\delta \in (0,\delta_0]$ there exists $\omega(\delta)>0$ such that
	\begin{equation}\label{upperbound}
		p(t) \leq  p(\tilde t) + \delta  \qquad \mbox{ for all } t \in [\tilde t, \tilde t + \omega(\delta)]\,.
	\end{equation}
	Towards this aim we will construct a suitable subsolution, more precisely, we are going to show that
	there exists a nonnegative, continuous function $W$ that depends only on the initial data $u_0$, with $W(\xi)>0$ for all $\xi>0$ and $W(0)=0$ such that
	\begin{equation}\label{subclaim}
		u(x,s) \geq W(x-p(s)), \quad \text{ for  } x \in (p(s),p(s)+\delta_0)\,, s \in [0,T]\;.
	\end{equation}
	Indeed,
	if \eqref{subclaim} holds we define for any $\tilde t \geq 0$
	the function $w$ as the solution to
	\begin{align}
		\partial_t{w}-\big((1-x^2) w' \big)'&=0 &&\quad \text{in } (p(\tilde t),p(\tilde t)+\delta_0)\times(\tilde t,T] \label{v1}\\
		w(x,\tilde t)&=W(x-p(\tilde t)) &&\quad \text{in } (p(\tilde t),p(\tilde t)+\delta_0)\label{v2}\\
		w(p(\tilde t),t)=w(p(\tilde t)+\delta_0,t)&=0 &&\quad \text{in } (\tilde t,\tilde t+\omega(\delta_0)] \label{v3}\;
	\end{align}
	and let $\tilde w:= w-(t-\tilde t)$.
	Equation \eqref{v1} yields
	\begin{align*}
		\partial_t{\tilde w }-\big((1-x^2)  {\tilde w}' \big)'=-1\leq \partial_t{u}-\big((1-x^2)  u' \big)'
	\end{align*}
	in $(p(\tilde t),p(\tilde t)+\delta_0)\times(\tilde t,T]$. Furthermore, $u(x,\tilde t)\geq W(x-p(\tilde t))$ for all $x\in(p(\tilde t),p(\tilde t)+\delta_0)$ by \eqref{subclaim} and  $u(p(\tilde t),t),u(p(\tilde t)+\delta_0,t)\geq 0= w(p(\tilde t),t)=w(p(\tilde t)+\delta_0,t)$ for all $t\in (\tilde t,T]$. Hence, we obtain by a comparison principle argument that
	\begin{equation}\label{v4}
		u(x,t)\geq  \tilde w \quad \text{in } [p(\tilde t),p(\tilde t)+\delta_0]\times[\tilde t,T].
	\end{equation}
	If we consider any $0<\delta\leq \tfrac{\delta_0}{2}$  it holds
	\begin{align*}
		u\geq \tilde w >0 \quad \text{in } [p(\tilde t)+\delta,p(\tilde t)+\delta_0-\delta) \times[\tilde t,\tilde t+\omega(\delta)]
	\end{align*}
	redefining $\omega(\delta)$ if necessary.
	Moreover, the monotonicity of $u$  yields that
	\begin{align*}
		u(x,t)>0 \quad \text{in } [p(\tilde t)+\delta,1]\times[\tilde t,\tilde t+\omega(\delta)]\;
	\end{align*}
	which proves the claim.
	
	\medskip
	We now proceed to the proof of \eqref{subclaim}.
	Due to Step $1$, we can fix a monotone increasing positive function $\omega:\R^+\to\R^+$ such that for all $s\in (-1,1]$, $t\in [-1,1)$
	\begin{equation}\label{Step1}
		p(t) \leq p(s)+\delta \quad\text{ if } 0\leq s-t\leq \omega(\delta)\,.
	\end{equation}
	To prove \eqref{subclaim} we argue differently depending on the value of $s$. 
	
	\medskip
	{\bf Case 1:} Consider $s >\omega(\delta_0)$ and let $\delta \in (0,\tfrac{\delta_0}{2}]$.
	By \eqref{Step1} we have 
	\begin{equation*}
		p(t) \leq p(s)+\delta
		\quad\text{ for all }t\in \big[s-\omega(\delta),s\big]\,.
	\end{equation*}
	Therefore, $u$ is a solution of
	\begin{equation*}
		\partial_t{u}-\big((1-x^2)  u'  \big)'=-1+\frac{g}{\lambda(t)} \quad \text{ in } [p(s)+\delta,p(s)+2\delta] \times (s-\omega(\delta),s]\;.
	\end{equation*}
	We observe that, by regularity of $u$ and standard embedding theorems, we can differentiate the above equation with respect to $x$.
	Then for $v=u'$ we obtain, using \eqref{g1d}, that
	\begin{equation*}
		\partial_t{v}-\big((1-x^2) v \big)''=\frac{g'}{\lambda(t)} \geq M\;  \quad \text{ in } [p(s)+\delta,p(s)+2\delta] \times (s-\omega(\delta),s]\;,
	\end{equation*}
	for some constant $M:=M({\vert \vert g\vert \vert}_{\infty}, {\vert \vert g'\vert \vert}_{\infty})>0$. 
	To construct a suitable subsolution let $V$ satisfy
	\begin{align}
		\partial_t{V}-\big((1-x^2)  V \big)''&\leq M, \quad \text{ in } [p(s)+\delta,p(s)+2\delta] \nonumber\\
		&\qquad\qquad\qquad\qquad \times (s-\omega(\delta),s]\;, \label{subeq}\\
		V(x,t^*_{\delta})&=0, \; \; \quad \text{ in } [p(s)+\delta,p(s)+2\delta]\;, \label{subtdata}\\
		V(p(s)+\delta,t)=V(p(s)+2\delta,t)&=0\,, \,\,\quad\text{ for } t \in  (s-\omega(\delta),s]\;. \label{subbcdata}
	\end{align}
	It is easy to verify that the function
	\begin{equation}
		V(x,t)= \frac{M}{2}(t-s+\omega(\delta))e^{-\mu\frac{t-s+\omega(\delta)}{\delta^2}}\cos
		\Big( \frac{\pi (x-(p(s)+\tfrac 3 2 \delta))}{\delta}\Big)\;,
	\end{equation}
	satisfies \eqref{subeq}-\eqref{subbcdata} if $\mu>0$ is sufficiently large.
	
	\medskip
	In particular we find that
	\begin{equation*}
		u'(x,s) \geq C \omega(\delta)e^{-\mu \frac{\omega(\delta)}{\delta^2}}=:F(\delta)\quad\text{ in }\Big[p(s)+\frac{5\delta}{4},p(s)+\frac{7\delta}{4}\Big]\,.
	\end{equation*}
	Therefore, since $\delta \leq \delta_0$ but otherwise arbitrary,
	we deduce
	\begin{equation*}
		u'(x,s) \geq  \tilde W (x-p(s)) \qquad \text{in } \big (p(s), p(s)+\delta_0\big ]
	\end{equation*}
	for some function $\tilde W$ which is positive on $\R_+$.
	Integrating this equation
	we find indeed that \eqref{subclaim} holds for $s>\omega(\delta_0)$.
	
	\medskip
	{\bf Case 2:} Next we investigate the case $s\leq \omega(\delta_0)$.
	Again by \eqref{Step1} we have
	\begin{equation*}
		p(t) \leq p(s)+\delta
		\quad\text{ for all }t\in \big[0,s\big]\,.
	\end{equation*}
	We find as above that  $v=u'$ solves
	\begin{equation*}
		\partial_t{v}-\big((1-x^2) v \big)''=\frac{ g'}{\lambda(t)}\geq M\;  \quad  [p(s)+\delta,p(s)+2\delta] \times (0,s]\;.
	\end{equation*}
	We construct a subsolution $V$ by solving the problem
	\begin{align}
		\partial_t{V}-\big((1-x^2)  V \big)''&= M, \quad  && \mbox{ in } [p(s)+\delta,p(s)+2\delta] \times (0,s]\;, \label{subeq2}\\
		V(x,0)&=u_0', \; \; \quad  && \mbox{ in }   [p(s)+\delta,p(s)+2\delta]\;, \label{subtdata2}\\
		V(p(s)+\delta,t)=V(p(s)+2\delta,t)&=0, \;\; \quad  &&\mbox{ in }(0,s] \label{subbcdata2}\;.
	\end{align}
	Since by assumption $u_0'> 0$ in $(p(0),1]\supset [p(s)+\delta,p(s)+2\delta]$, we conclude that $V(x,s)>\phi(\delta, u_0)$ in $[p(s)+ \tfrac 5 4 \delta, p(s) + \tfrac 7 4 \delta]$ and as above we finally conclude that there exists a positive function $W_2$ such that
	\begin{equation}\label{subclaimP3}
		u(x,s)\geq W_2(x-p(s)), \quad \text{in } \;[p(s),p(s)+\delta_0 ]
	\end{equation}
	for $s \in [0,\omega(\delta_0)]$ from which \eqref{subclaim} follows with  $W:=\min\{W_1,W_2\}>0$.
	
	\medskip
	{\bf Step 3:}  Since $\omega(\delta)$ does not depend on $\tilde t$, the estimates \eqref{lowerbound} and \eqref{upperbound} imply the uniform continuity from the right therefore that $p$ is a continuous function.

	
	\section{An oscillatory solution to the classical parabolic obstacle problem}
	\label{sec:obsc-obs}
	
	In this section we will provide a rigorous construction of an oscillatory solution to the classical parabolic obstacle problem.
	
	\medskip
	The classical parabolic obstacle problem in one space dimension and in the whole space can be formulated as follows. 
	We say that a function $u\colon \R \times (0,T)\to [0,\infty)$ is a solution to the obstacle problem with data $u_0$ if  $u \in W^{2,1}_p( \R\times (0,T))$  for all $p \in [1,\infty)$ and if it satisfies
	\begin{align}
		\partial_t u - u'' &= - H(u) \qquad \text{ a.e.~in } \R \times (0,T) \,\label{obst1}\\
		u(\cdot, 0)& = u_0 \qquad \text{ in a suitable sense. }
		\label{obst3}
	\end{align}
	
	In order to prepare the construction of an oscillatory solution to a nonlocal obstacle problem 
	in Section \ref{sec:nonlocal-obs-R}  we also consider the problem with a rescaled right-hand side:
	For $\lambda>0$ we denote by $u_\lambda$ the solution of
	\begin{align}
		\partial_t u_\lambda - u_\lambda'' &= - \lambda^3 H(u_\lambda) \qquad \mbox{ a.e.~in } \R \times (0,T) \,\label{4eq:obst1-lambda}\\
		u_\lambda(\cdot, 0)& = u_0 \qquad \text{ in a suitable sense. }
		\label{5eq:obst3-lambda}
	\end{align}
	
	We will need the following consequence of the maximum principle, which is a variant of Lemma \ref{L.degeneracy}.
	
	\begin{lemma} \label{4lem:maxprinciple}
		Consider any $x_0\in\R$, $t_0>0$, $\theta>0$ and $0<\varrho<\sqrt{t_0}$.
		Then for any $u\in W^{2,1}_1\big(B(x_0,\varrho)\times (t_0-\varrho^2,t_0)\big)$ with
		\begin{align*}
			\partial_t u -\Delta u &\leq -\theta H(u),\quad
			0\leq u\leq \frac{\theta}{3}\varrho^2\qquad
			\text{ in }B(x_0,\varrho)\times (t_0-\varrho^2,t_0)
		\end{align*}
		we have $u(x_0,t_0)=0$.
	\end{lemma}
	
	\begin{proof}
		Assume $u(x_0,t_0)>0$ and define
		\begin{equation*}
			w(x,t) := \frac{\theta}{3}|x-x_0|^2 + \frac{\theta}{3}(t_0-t)
		\end{equation*}
		Then
		\begin{align*}
			\partial_t w -\Delta w = -\theta\quad\text{ in }B(x_0,\varrho)\times (t_0-\varrho^2,t_0),
		\end{align*}
		and $w(x,t_0-\varrho^2)\geq \frac{\theta}{3}\varrho^2$ for all $x\in\R$ and $w(x,t)\geq \frac{\theta}{3}\varrho^2$ for all $x\in\partial B(x_0,\varrho)$, $t\in (t_0-\varrho^2,t_0)$.
		
		We apply the weak  maximum principle in $\big(B(x_0,\varrho)\times (t_0-\varrho^2,t_0)\big)\cap \{u>0\}$ and deduce $u\leq w$.
		This yields $u(x_0,t_0)\leq 0$, a contradiction to our assumption.
	\end{proof}
	
	Our main result on the oscillatory behavior for the classical parabolic obstacle problem is the following.
	\begin{theorem}\label{main.ob}
		There exist initial data $u_0 \in {\cal M}_+(\R)$
		with compact support in $(0,1)$ such that the solution to \eqref{obst1}, \eqref{obst3} satisfies
		\begin{equation}\label{oscillationOb}
			\liminf \limits_{t\to 0^+} |\{u(\cdot,t)>0\}| < \limsup\limits_{t\to 0^+} |\{u(\cdot,t)>0\}|\;.
		\end{equation}
	\end{theorem}
	
	Actually, we will prove the following stronger statement that will be needed below in Section \ref{sec:nonlocal-obs-R} for the nonlocal case.
	
	\begin{theorem}\label{4thm:main}
		For any $\lambda_0>0$ there exists $0<\kappa<\lambda_0$ with the following property:
		For any $0<\eta<1$ there exist initial data $u_0 \in {\cal M}_+(\R)$
		with compact support in $(0,\eta)$ such that for any $\lambda,\mu$ in $[\lambda_0-\kappa,\lambda_0+\kappa]$ the solutions $u_\lambda$, $u_\mu$ to \eqref{4eq:obst1-lambda}, \eqref{5eq:obst3-lambda} satisfy
		\begin{equation}\label{4eq:oscillation}
			\liminf \limits_{t\to 0^+} |\{u_\lambda(\cdot,t)>0\}| < \limsup\limits_{t\to 0^+} |\{u_\mu(\cdot,t)>0\}|\;.
		\end{equation}
	\end{theorem}
	
	Before proceeding to the proof of this theorem, we first consider the problem \eqref{obst1}-\eqref{obst3} with initial data
	\begin{equation}\label{obst4}
		u_0 = \delta_0
	\end{equation}
	where $\delta_0$ denotes the Dirac distribution in $x=0$.
	
	The solution of this particular problem is the key building block to construct initial data $u_0$ as in Theorem \ref{main.ob} and Theorem \ref{4thm:main}.
	
	\begin{proposition}\label{P.solproperties}
		There exists a unique solution $U\in W^{2,1}_p\big(\R\times(\delta,\infty)\big)$ for all $1\leq p<\infty$, $\delta>0$ that satisfies \eqref{obst1} almost everywhere in $\R\times (0,\infty)$ and \eqref{obst4} in the sense of
		\begin{equation}
			\big|U(x,t)-\Phi(x,t)\big| \leq Ct\quad\text{ for all }x\in\R,\,t>0
			\label{5eq:initial:data}
		\end{equation}
		for some $C$ independent of $x$, where $\Phi$ denotes the heat kernel $\Phi(x,t)=\frac{1}{\sqrt{4 \pi t}}e^{-\frac{|x|^2}{4t}}$.
		
		Moreover the solution satisfies the following:
		\begin{enumerate}
			\item \label{it:first}  
			There exists $T^*>0$ such that $U(x,t)=0$ for any $t\geq T^*$ and for all $x\in \R$. 
			\item \label{it:second} 
			The solution $U(x,t)$ is symmetric, that is $U(x,t)=U(-x,t)$ for all $x\in \R$ and $t >0$.
			\item \label{it:third}
			The first derivative satisfies $\partial_{x} U(x,t) \leq 0$ for all $x\geq 0$ and $t>0$.
			\item \label{it:fourth}
			There exist two continuous function $\ell\colon(0,T^*)\rightarrow\R$ and $L\colon (0,T^*)\rightarrow\R$
			such that for each $t\in(0,T^*)$ we have 
			\begin{equation*}
				\left\{ |x| \leq\ell(t) \right\}  
				\subset\left\{U(x,t)>0\right\}  
				\subset\left\{|x|\leq L(t)\right\}.
			\end{equation*}
			Moreover, it holds
			\begin{equation*}
				\lim_{t\rightarrow0^{+}}\frac{\ell(t)  }{\sqrt{ 6t \ln\big (\frac{1}{t}\big )}}=\lim_{t\rightarrow0^{+}}\frac{L(t)		}{\sqrt{ 6t \ln\big (\frac{1}{t}\big )}}=1.
			\end{equation*}
		\end{enumerate}
	\end{proposition}
	
	\begin{proof}
		Existence and uniqueness follows from \cite[Theorem 2.1]{BF76}.
		
		We proceed to the proof of the additional properties.
		\begin{enumerate}
			\item Item \ref{it:first} follows from \cite[Theorem 3.1]{BF76}.
			\item Item \ref{it:second} follows from the invariance of the equation and the initial data under the transformation $x \mapsto -x$.
			\item The proof of item \ref{it:third} is analogous to the one of Lemma \ref{L.1} and we omit it here.
			\item Left-hand side inclusion in \ref{it:fourth}:  
			Since $H(u) \leq 1$ the function $(x,t)\mapsto \Phi(x,t)-t$ is a subsolution and we deduce from the maximum principle that
			\begin{equation*}
				U(x,t) \geq \frac{1}{2\sqrt{\pi t}}e^{-\frac{|x|^2}{4t}}-t \quad\text{ for all }x\in \R,\,t>0. 
			\end{equation*}
			This implies in particular that $\bigg \{\frac{1}{2\sqrt{\pi t}}e^{-\frac{|x|^2}{4t}}>t \bigg \} \subset \{ U(x,t)>0\}\;$.
			Then it is easily calculated that $\frac{1}{2\sqrt{\pi t}}e^{-\frac{|x|^2}{4t}}>t$ is equivalent to
			$ |x|^2 < 6t \ln(\frac{1}{(4\pi)^{1/3}t})$ and the claim follows.
			\item Right-hand side inclusion in \ref{it:fourth}:  
			By comparison principle we deduce
			\begin{equation}
				U(x,t) \leq \Phi(x,t) =\frac{1}{\sqrt{4 \pi t}}e^{-\frac{|x|^2}{4t}}\quad\text{ for all }x\in\R,\,t>0.
				\label{4eq:uleqPhi}
			\end{equation}
			For some $(x_0,t_0) \in [0,\infty)\times (0,\infty)$ and $\varrho>0$ that we specify below in \eqref{4eq:4.8a} we consider the rectangle
			\begin{equation}\label{rectangle}
				\mathcal R:=[x_0-\varrho,x_0+\varrho]\times[t_0-\varrho^2,t_0]\;.
			\end{equation}
			Exploiting \eqref{4eq:uleqPhi}, we obtain that
			\begin{align}
				U(x,t) \leq  \frac{1}{2\sqrt{\pi (t_0-\varrho^2)}}e^{-\frac{|x_0-\varrho|^2}{4t_0}} \quad \text{ for all } (x,t) \in \mathcal R\;.
				\label{C3}
			\end{align}
			If
			\begin{equation}
				|x_0-\varrho|^2\geq 12\alpha\varrho^2 (-\ln\varrho)
				\quad\text{ and }\quad t_0=\alpha\varrho^2
				\label{4eq:4.8a}
			\end{equation}
			for some $\alpha>1$ to be chosen below, we deduce from \eqref{C3} that for all $(x,t) \in \mathcal R$
			\begin{align*}
				U(x,t) \leq  \frac{1}{\sqrt{4\pi(\alpha -1)}}\varrho^2.
			\end{align*}
			Choosing $\alpha>9\{1,T^*\}$ sufficiently large Lemma \ref{4lem:maxprinciple} yields that $U(x_0,t_0)=0$. 
			Hence, $U(x,t)=0$ for all $x\in\R$, $t>0$ with
			\begin{equation*}
				|x| \geq \sqrt{6t(-\ln t) + 6t\ln\alpha} +\sqrt{\frac{t}{\alpha}}\,=:\,L(t)\,.
			\end{equation*}
			This implies the second inclusion in \ref{it:fourth}.
		\end{enumerate}
	\end{proof}
	
	\begin{figure}
		\centering
		\includegraphics[width=0.5\linewidth]{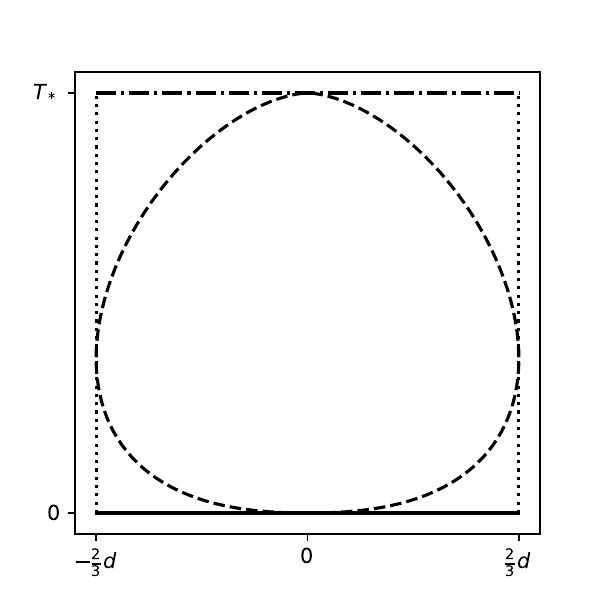}
		\caption{Main building block of the construction. The boundary of the support of $U$ is indicated by the oval line.}
	\end{figure}
	
	\begin{remark}
		\label{4rem:U-lambda}
		Let $U$ denote the solution of \eqref{obst1}, \eqref{obst4} and let for $\lambda>0$ fixed
		\begin{equation*}
			U_\lambda(x,t)=\lambda U(\lambda x,\lambda^2t),\,x\in\R,t>0.
		\end{equation*}
		Then $U_\lambda$ satisfies \eqref{4eq:obst1-lambda}, $U(\cdot,0)=\delta_0$ and
		\begin{equation*}
			\left\{ |x| \leq\ell_\lambda(t) \right\}  
			\subset\left\{U_\lambda(x,t)>0\right\}  
			\subset\left\{|x|\leq L_\lambda(t)\right\}.
		\end{equation*}
		with $\ell_\lambda(t)=\frac{1}{\lambda}\ell(\lambda^2t)$, $L_\lambda(t)=\frac{1}{\lambda}L(\lambda^2t)$.
		In particular
		\begin{equation*}
			\lim_{t\rightarrow0^{+}}\frac{\ell_\lambda(t)}{\sqrt{ 6t \ln\big (\frac{1}{t}\big )}}
			=\lim_{t\rightarrow0^{+}}\frac{\ell_\lambda(t)}{\sqrt{ 6t \ln\big (\frac{1}{\lambda^2t}\big )}}
			=1=\lim_{t\rightarrow0^{+}}\frac{L_\lambda(t)}{\sqrt{ 6t \ln\big (\frac{1}{\lambda^2t}\big )}}
			=\lim_{t\rightarrow0^{+}}\frac{L_\lambda(t)}{\sqrt{ 6t \ln\big (\frac{1}{t}\big )}}.
		\end{equation*}
	\end{remark}
	
	\begin{proof}[Proof of Theorem \ref{4thm:main}]
		Let $U$ denote the solution of \eqref{obst1}, \eqref{obst4}.
		By a scaling argument it is sufficient to consider the case $\lambda_0=1$.
		Let $\eta>0$ be arbitrarily prescribed.
		
		We are going to construct a solution $u$ of \eqref{obst1} with oscillatory support such $u(\cdot,0)$ is supported in the interval $(0,\eta)$.
		We will use an iterative procedure with $U$ as main building block.
		
		Let 
		\begin{equation}\label{defn-d}
			d:=\frac{3}{2}\sup_{t >0} \{ x>0\,|\, U(x,t)>0\}.
		\end{equation}
		By Remark \ref{4rem:U-lambda} the support of the solution $U_\lambda$ of \eqref{4eq:obst1-lambda}, \eqref{obst4} satisfies
		\begin{equation}
			\{U_\lambda(\cdot,t)>0\} \subset (-d,d) \quad\text{ for all }t>0,\,\lambda> \frac{2}{3}.
			\label{4eq:d-support}
		\end{equation}
		Increasing the value of $T^*$ from Proposition \ref{P.solproperties} item \ref{it:first} by a factor $\frac{9}{4}$ we have
		\begin{equation*}
			U_\lambda(\cdot,t)=0 \quad\text{ for all }t\geq T^*,\,\lambda>\frac{2}{3}.
		\end{equation*}
		Next, for a parameter $\theta \in (0,1)$ we define the rescaled solution
		\begin{equation}\label{scaled-sol}
			u_{\theta,\lambda}(x,t):= \lambda\theta^2 U \Big ( \frac{\lambda x}{\theta}, \frac{\lambda^2t}{\theta^2}\Big)
			=\theta^2U_\lambda(\frac{x}{\theta}, \frac{t}{\theta^2})\,.
		\end{equation}
		Furthermore we deduce from Proposition \ref{P.solproperties} and Remark \ref{4rem:U-lambda} that there exist constants $C_1>0$ and $0<\kappa_0<\frac{1}{3}$ such that
		\begin{equation}
			\{U_\lambda(\cdot,t)>0\} \subset \Big( -C_1 \sqrt{t \ln  t^{-1}},C_1 \sqrt{t \ln  t^{-1}} \Big)
			\quad\text{ for all }0<t<\frac{1}{2},\,|1-\lambda|<\kappa_0\,.
			\label{4eq:15a}
		\end{equation}
		This implies in particular that for any $x_0\in\R$, $0<t<\frac{1}{2}$, $|1-\lambda|<\kappa_0$
		\begin{equation}\label{uthetazero}
			u_{\theta,\lambda}(x-x_0,t)= 0 \qquad
			\mbox{ if } \quad \frac{|x-x_0|}{\theta} \geq C_1\sqrt{ \frac{t}{\theta^2} \ln \frac{\theta^2}{t}}\,, \quad
			\mbox{ that is if} \quad |x-x_0| \geq C_1 \sqrt{t \ln \frac{\theta^2}{t}}\,.
		\end{equation}
		
		We are now going to define iteratively a set of points that represents the atoms of the measure that we choose as initial datum.
		
		Consider a sequence $(\theta_n)_{n\in\N}$ that converges strictly monotone to zero and that we will specify later.
		
		In a first step we iteratively define sets of points $x^{(n)}_{j_1,\cdots,j_n}$, $n\in\N$, via
		\begin{align*}
			x^{(1)}_{j_1} & = 4 d \theta_1 j_1 \,, \qquad \qquad
			j_1 \in I_1:=\Big\{1,2,\cdots,\Big \lfloor \frac{1}{4 d \theta_1}\Big \rfloor -1\Big \}, \\
			x^{(2)}_{j_1,j_2} & = x^{(1)}_{j_1}+ 4 d \theta_2 j_2 \,, \qquad
			j_1\in I_1,\, j_2 \in I_2:=\Big \{1,2,\cdots,\Big \lfloor \frac{\theta_1}{\theta_2}\Big \rfloor -1\Big \}, \\
			\vdots & \qquad \vdots \\
			x^{(n)}_{j_1,j_2,\cdots,j_n} & = x^{(n-1)}_{j_1,\cdots,j_{n-1}}+ 4 d \theta_n j_n \,, \qquad
			j_1\in I_1,\dots,\,j_n \in I_n:=\Big \{1,2,\cdots,\Big \lfloor \frac{\theta_{n-1}}{\theta_{n}}\Big\rfloor  -1\Big \}.
		\end{align*}
		We note that the number $Z_n$ of points on the level $n$ is of the order $\big( 4 d \theta_n\big)^{-1}$.
		
		Next we consider index sets $I^*_n\subset I_n$, $n\in\N$ that we will specify later, and define a corresponding solution $u^{(n)}_\lambda$, $n\in\N$ of \eqref{obst1} by
		\begin{align}
			u^{(1)}_\lambda&= \sum_{j_1 \in I^*_1} u_{\theta_1,\lambda}^{j_1},&\qquad u_{\theta_1,\lambda}^{j_1}(x,t)&=u_{\theta_1,\lambda}\big(x-x^{(1)}_{j_1},t\big)\,,\notag \\
			u^{(2)}_\lambda&= \sum_{j_1 \in I^*_1}\sum_{j_2 \in I^*_2} u_{\theta_2,\lambda}^{j_1,j_2},&\qquad u_{\theta_2,\lambda}^{j_1,j_2}(x,t)&=u_{\theta_2,\lambda}\big(x-x^{(2)}_{j_1,j_2},t\big)\,,\notag \\
			\vdots & \qquad \vdots \notag\\
			u^{(n)}_\lambda &=  \sum_{j_1\in I^*_1} \sum_{j_2 \in I^*_2} \cdots \sum_{j_n \in I^*_n} u_{\theta_n,\lambda}^{j_1,\cdots,j_n},&\quad u_{\theta_n,\lambda}^{j_1,\cdots,j_n}(x,t)&=u_{\theta_n,\lambda}\big(x-x^{(n)}_{j_1,\cdots,j_n},t\big)\,.
			\label{n-supp-def}
		\end{align}
		Finally we define
		\begin{align}\label{udef}
			u_\lambda(x,t):=\sum_{n \in \N} u^{(n)}_\lambda(x,t)\,.
		\end{align}
		
		\begin{figure}
			\centering
			\includegraphics[width=0.6\linewidth]{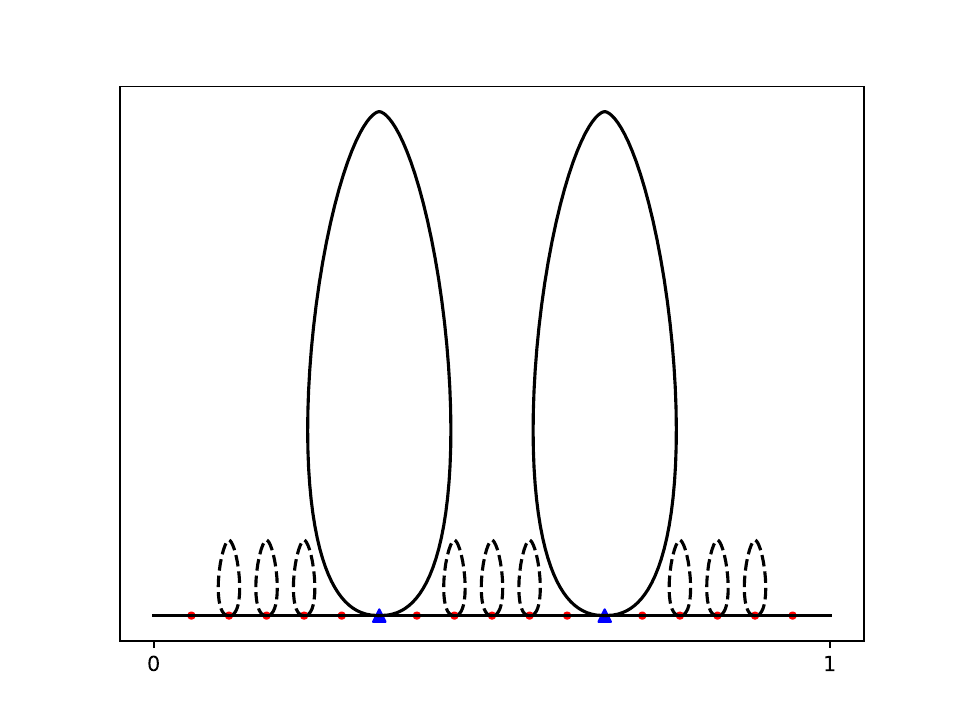}
			\caption{Points on the level $n$ (circles) and $n+1$ (triangles)}
		\end{figure}

		We will choose the index sets $I^*_n$ below such that the supports of the different building blocks do not overlap.
		This in particular implies that the infinite sum in \eqref{udef} converges and that $u_\lambda$ is a solution of \eqref{obst1} with initial datum
		\begin{equation*}
			u_0:=\sum_{n \in \N}\sum_{j_1\in I^*_1} \sum_{j_2 \in I^*_2} \cdots \sum_{j_n \in I^*_n} \theta_n^3\delta_{x^{(n)}_{j_1,\cdots,j_n}}.
		\end{equation*}
		We observe that $u_0$ is a finite Radon measure if
		\begin{equation}
			\sum_{n\in\N} (4d\theta_n)^{-1}\theta_n^3<\infty,\text{ which is guaranteed by }
			\sum_{n\in\N}\theta_n^2<\infty.
			\label{4eq:init-theatn}
		\end{equation}
		
		In the following we always assume $0<t<\frac{1}{2}$.
		We specify the choices of the index sets $I^*_n$ such that $u_\lambda$ is well-defined and such that the size of supports of $u_\lambda(\cdot,t)$ is oscillatory for $t\searrow 0$.
		
		First, we notice that due to the definition of $d$, see \eqref{4eq:d-support}, and the choice of the spacing of the $x^{(n)}$ the supports $u_{\theta_n,\lambda}^{j_1,\dots,j_n}$ do not overlap.
		
		We will now choose the index set $I^*_n\subset I_n$ such that the support of  $u_{\theta_n,\lambda}$ does not overlap with any of the previous levels.
		We therefore consider a point $x^{(n-1)}_{j_1,\dots,j_{n-1}}$ with $j_k\in I^*_k$ for $k=1,\dots,j-1$.
		In order to exclude that the support of $u_{\theta_n,\lambda}^{j_1,\cdots,j_n}$ overlaps with the support of $u_{\theta_{n-1},\lambda}^{j_1,\cdots,j_{n-1}}$ by \eqref{4eq:15a} it is sufficient to guarantee that
		\begin{equation*}
			4 d \theta_{n} j_n - d\theta_n > C_1 \sqrt{ T^* \theta_n^2 \ln \Big( \frac{\theta_{n-1}^2}{T^*\theta_n^2}\Big)},
		\end{equation*}
		that is
		\begin{equation*}
			j_n > \frac{1}{4} + \frac{C_1}{4d} \sqrt{ T^*  \ln \Big( \frac{\theta_{n-1}^2}{T^*\theta_n^2}\Big)}
		\end{equation*}
		We notice that this is ensured if
		\begin{equation}\label{cond-in}
			j_n > \frac{C_1}{2d} \sqrt{ T^*  \ln \Big( \frac{\theta_{n-1}^2}{T^*\theta_n^2}\Big)}
		\end{equation}
		and the sequence $(\theta_n)_n$ satisfies
		\begin{equation}\label{thetan-1}
			\frac{\theta_n}{\theta_{n-1}} < \frac{1}{\sqrt{T^*}}\exp \Big( - \frac{d^2}{2T^* C_1^2} \Big)=:C_2\,.
		\end{equation}
		Similarly, the conditions \eqref{thetan-1} and
		\begin{equation}\label{cond-in2}
			j_n < \frac{\theta_{n-1}}{\theta_{n}} - \frac{C_1}{2d} \sqrt{ T^*  \ln \Big( \frac{\theta_{n-1}^2}{T^*\theta_n^2}\Big)}
		\end{equation}
		exclude that the support of $u_{\theta_n}^{j_1,\cdots,j_n}$ overlaps with the support of $u_{\theta_{n-1}}^{j_1,\cdots,j_{n-1}+1}$.
		
		Therefore, we can choose $I^*_n \subset I_n$ with a number of indices 
		\begin{equation*}
			|I^*_n| \geq |I_n|-\frac{C_1}{d} \sqrt{ T^*  \ln \Big( \frac{\theta_{n-1}^2}{T^*\theta_n^2}\Big)}.
		\end{equation*}
		The number $Z_n^*$ of indices $(j_1,\dots,j_n)$ with $j_k\in I_k^*$ for all $k=1,\dots,n$ can thus be estimated from below by
		\begin{align*}
			Z_n^* &\geq \Big(\frac{1}{4d\theta_1}-2\Big)\cdot \left(\frac{\theta_1}{\theta_2}-2-\frac{C_1}{d} \sqrt{ T^*  \ln \Big( \frac{\theta_1^2}{T^*\theta_2^2}\Big)}\right)\ldots \cdot \left(\frac{\theta_{n-1}}{\theta_n}-2-\frac{C_1}{d} \sqrt{ T^*  \ln \Big( \frac{\theta_{n-1}^2}{T^*\theta_n^2}\Big)}\right)\\
			&\geq \frac{1}{4d\theta_n}\Pi_{j=1}^n(1-\eps_j),
		\end{align*}
		with
		\begin{equation*}
			\eps_1=8d\theta_1,\quad
			\eps_j=\frac{\theta_j}{\theta_{j-1}}\Big(2+\frac{C_1}{d} \sqrt{ T^*  \ln \Big( \frac{\theta_{j-1}^2}{T^*\theta_j^2}\Big)}\Big),\quad j\geq 2.
		\end{equation*}
		This yields
		\begin{equation}
			Z_n^* \geq \frac{1}{4d\theta_n}(1-\sum_{j=1}^n\eps_j)
			\geq \frac{2}{9d\theta_n} \quad\text{ if }\quad\sum_{j=1}^\infty \eps_j<\frac{1}{9}.
			\label{4eq:cond-sum-epsj}
		\end{equation}
		
		Next, we define a sequence of times along which the support of the solution oscillates. 
		By Proposition \ref{P.solproperties} there exists $0<T_1<T^*$ and $0<\kappa<\kappa_0$ such that
		\begin{equation}
			|\{U_\lambda(\cdot,T_1)>0\}|\geq l(T_1)=: d_1 > 0\quad\text{ for all }\lambda\in (1-\kappa,1+\kappa).
			\label{4eq:T1}
		\end{equation}
		We then define $t_n:=T_1\theta_n^2$, which implies for any $j_1,\dots,j_n$ with $j_k\in I_k^*$, $k=1,\dots,n$ that for all $|1-\lambda|<\kappa$
		\begin{equation}\label{tndef}
			\Big| \{ u^{(n)}_\lambda(\cdot,t_n) >0\} \cap \big [ x^{(n)}_{j_1,\cdots, j_n} - 2 d \theta_n, x^{(n)}_{j_1,\cdots,j_n} + 2 d \theta_n \big ]\Big| \geq d_1 \theta_n\,.
		\end{equation}
		Again by Proposition \ref{P.solproperties} we can choose $0<\tilde T_1\ll T_1$ such that 
		\begin{equation*}
			|\{U(\cdot,\tilde T_1)>0\}|\leq \frac{d_1}{8}\,,
		\end{equation*}
		hence for all $0<\kappa<\kappa_0$ sufficiently small
		\begin{equation}
			|\{U_\lambda(\cdot,\tilde T_1)>0\}|\leq \frac{d_1}{4}\quad\text{ for all }\lambda\in (1-\kappa,1+\kappa).
			\label{4eq:tildeT1}
		\end{equation}
		We define $\tilde t_n:=\tilde T_1\theta_n^2$, which implies for any $j_1,\dots,j_n$ with $j_k\in I_k$, $k-1,\dots,n$ and any $\lambda\in (1-\kappa,1+\kappa)$ that 
		\begin{equation} \label{ttildendef}
			\Big| \{ u^{(n)}_\lambda(x,{\tilde t}_n) >0\} \cap \big [ x^{(n)}_{j_1,\cdots, j_n} - 2 d \theta_n, x^{(n)}_{j_1,\cdots,j_n} + 2 d \theta_n \big ]\Big| \leq \frac{d_1}{4} \theta_n\,.
		\end{equation}
		We ensure that there is no overlap with the support of $u^{(k)}_\lambda(\cdot,\tilde t_n)=0$ for any $k\geq n+1$, which is guaranteed by
		\begin{equation}\label{thetan-2}
			T^*\theta_{n+1}^2 < \tilde t_n\,,\text{ thus by }\quad\frac{\theta_{n+1}}{\theta_n}< \sqrt{\frac{\tilde T_1}{T^*}}.
		\end{equation}
		We assume $\tilde T_1<1$, set $\theta_0=\frac{1}{4d}$ and estimate the overlap with the support of $u^{(k)}_\lambda(\cdot,\tilde t_n)=0$ for any $1\leq k\leq n-1$ and $|1-\lambda|<\kappa$ by
		\begin{align}
			\Big| \big\{u^{(k)}_\lambda(\cdot,\tilde t_n)>0\big\}\Big|
			&\leq \big|I_k^*\big|2C_1\sqrt{\tilde t_n\ln \Big(\frac{\theta_k^2}{\tilde t_n}\Big)}
			\nonumber\\
			&\leq \frac{\theta_{k-1}}{\theta_k}2C_1\theta_n \sqrt{\tilde T_1\ln \Big(\frac{\theta_k^2}{\tilde T_1\theta_n^2}\Big)}
			\nonumber\\
			&\leq 2C_1\theta_{k-1} \sqrt{\frac{\tilde T_1\theta_n^2}{\theta_k^2}\ln \Big(\frac{\theta_k^2}{\tilde T_1\theta_n^2}\Big)}
			\nonumber\\
			&\leq 2C_1\theta_0\sqrt{\frac{\tilde T_1\theta_n^2}{\theta_{n-1}^2}\ln \Big(\frac{\theta_{n-1}^2}{\tilde T_1\theta_n^2}\Big)}
			\nonumber\\
			&\leq 2C_1\theta_0\frac{\theta_n}{\theta_{n-1}}\sqrt{\tilde T_1\ln \Big(\frac{\theta_{n-1}^2}{\tilde T_1\theta_n^2}\Big)}.
			\label{4eq:sum-thetaj}
		\end{align}
		This yields
		\begin{align*}
			\Big| \big\{\sum_{k=1}^{n-1}u^{(k)}_\lambda(\cdot,\tilde t_n)>0\big\}\Big|
			&\leq 2C_1\theta_0\frac{n\theta_n}{\theta_{n-1}}\sqrt{\tilde T_1\ln \Big(\frac{\theta_{n-1}^2}{\tilde T_1\theta_n^2}\Big)}
			\quad\text{ for all }\lambda\in (1-\kappa,1+\kappa).
		\end{align*}
		Together with \eqref{ttildendef} we can estimate the size of support at $t=\tilde t_n$ from above by
		\begin{equation}
			\big| \{ u_\lambda(\cdot,{\tilde t}_n)>0\}\big | 
			\leq \frac{d_1}{4}\theta_n\Big\lfloor\frac{1}{4d\theta_n}\Big\rfloor + 2C_1\theta_0\frac{n\theta_n}{\theta_{n-1}}\sqrt{\tilde T_1\ln \Big(\frac{\theta_{n-1}^2}{\tilde T_1\theta_n^2}\Big)}
			\leq \frac{1}{9}\frac{d_1}{d}
			\label{4eq:est-tilde-tn}
		\end{equation}
		if
		\begin{equation}
			2C_1\theta_0\frac{n\theta_n}{\theta_{n-1}}\sqrt{\tilde T_1\ln \Big(\frac{\theta_{n-1}^2}{\tilde T_1\theta_n^2}\Big)} \leq \frac{1}{18}\frac{d_1}{d}\quad\text{ for all }n\in\N.
			\label{4eq:cond-tilde-tn}
		\end{equation}
		
		We next obtain a suitable bound from below for the support at time $t=t_n$.
		By \eqref{4eq:cond-sum-epsj}, \eqref{tndef} the support of $u^{(n)}_\lambda(\cdot,t_n)$ is estimated from below by
		\begin{equation}
			\big| \{ u^{(n)}_\lambda(\cdot,t_n)>0\}\big| \geq d_1\theta_n Z_n^*
			\geq d_1\theta_n \frac{2}{9d\theta_n} = \frac{2}{9}\frac{d_1}{d},
			\label{measure-n}
		\end{equation}	
		if the condition on $(\theta_n)_n$ in \eqref{4eq:cond-sum-epsj} holds.
		
		We remark, that all conditions \eqref{thetan-1}, \eqref{4eq:cond-sum-epsj}, \eqref{thetan-2} and \eqref{4eq:cond-tilde-tn} can be satisfied if we only choose the sequence $(\theta_n)_n$ with a sufficiently strong decay to zero.
		
		To summarize, for any $\lambda\in (1-\kappa,1+\kappa)$ we have constructed a solution $u_\lambda$ to the obstacle problem on the real line with support in $(0,1)$ and two sequences $t_n, \tilde t_n \to 0$ with 
		\begin{equation*}
			\limsup \limits_{t_n\to 0} \vert u_\lambda(x,t_{n})\vert \geq \frac{2}{9} \frac{d_1}{d}\quad \text{and} \quad \liminf \limits_{\tilde t_n\to 0} \vert u_\lambda(x,\tilde t_{n})\vert \leq \frac{1}{9}\frac{d_1}{d}\,.
		\end{equation*}
		This in particular proves Theorem \ref{4thm:main} with $\eta=1$.
		
		The construction shows that for any $\eta>0$ the property \eqref{4eq:oscillation} remains valid if we set $u_0$ to zero outside $(0,\eta)$.
	\end{proof}

	\section{Oscillatory solutions for a nonlocal obstacle problem on the real line}
	\label{sec:nonlocal-obs-R}
	
	In this section we construct oscillatory solutions with compact support to the nonlocal obstacle problem 
	\begin{align}
		\partial_t u - u'' &= -\Big(1-\frac{g}{\lambda}\Big)H(u)\quad&\text{ in }I\times (0,\infty),
		\label{5eq:ob1}\\
		u &\geq 0\quad&\text{ in }I\times (0,\infty), \label{5eq:ob2}\\
		g &\leq \lambda(t) \quad\text{ a.e.~in } \{u=0\}&\text{ for almost every }t>0,
		\label{5eq:ob3}\\
		\lambda(t) &= \fint_{\{u(\cdot,t)>0\}} g(x)\,dx \quad&\text{ for almost every }t>0,
		\label{5eq:ob4}\\
		u(\cdot,0) &= u_0\quad&\text{ in }I
		\label{5eq:ob5}
	\end{align}
	with $I=\R$.
	The solution $u(\cdot,t)$ will have compact support in $(-5,5)$ for all sufficiently small times and will be constructed by considering \eqref{5eq:ob1}-\eqref{5eq:ob5} with $I=(-5,5)$ and with a Neumann boundary condition
	\begin{equation}
		u'(\pm 5,t) = 0 \quad\text{ for all }t>0.
		\label{5eq:ob4a}
	\end{equation}
	For simplicity we choose a particular $g\in C^2([-5,5])$, satisfying $g_{\min}:=\frac{1}{12}\leq g\leq \frac{6}{7} =: g_{9}$ and
	\begin{align}
		g=g_{9}\,\text{ in }[-3,-2],\quad
		g = g_{\min}\,\text{ in } [-5,5]\setminus (-\frac{7}{2},-\frac{3}{2}).
	\end{align}
	We consider $u_0=u_0^-+u_0^+$ with $u_0^-\in C^2(-5,5)$, $\{u_0^->0\}=(-4,-1)$ and $u_0^+$ a finite nonnegative Radon measure supported in $(0,\eta)$ for some $\eta>0$ that will be chosen below as a rescaling of the initial datum constructed in Section \ref{sec:obsc-obs}.
	
	\begin{figure}
		\centering
		\includegraphics[width=0.6\linewidth]{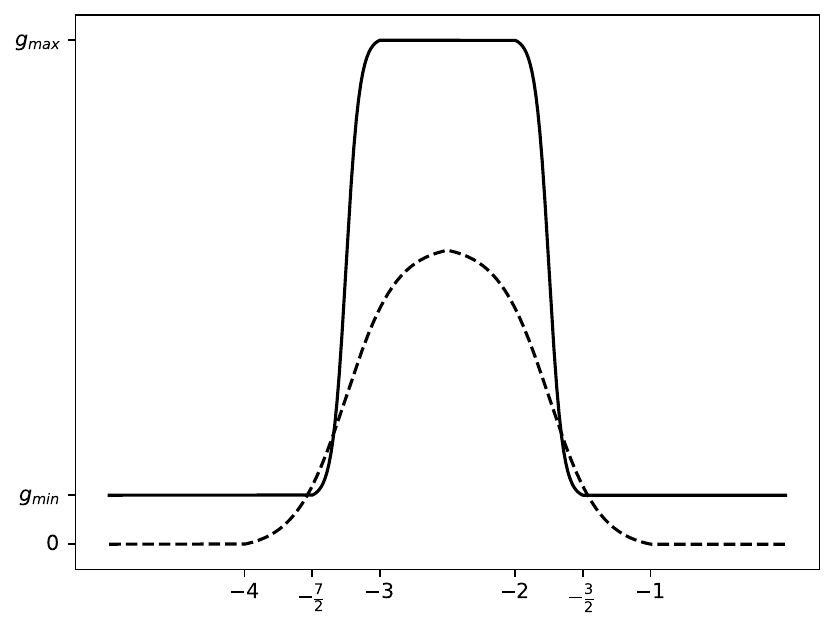}
		\caption{Graph of $g$ (solid) and of $u_0^-$ (dashed)}
	\end{figure}
	
	Our main result in this section is the following.
	\begin{theorem}\label{5thm:main}
		There exist initial data $u_0 \in {\cal M}_+\big((-5,5)\big)$ such that the solution to \eqref{5eq:ob1}-\eqref{5eq:ob5} with $I=\R$ satisfies
		\begin{equation}\label{5eq:oscillation}
			\liminf \limits_{t\to 0^+} |\{u(\cdot,t)>0\}| < \limsup\limits_{t\to 0^+} |\{u(\cdot,t)>0\}|
		\end{equation}
		and
		\begin{equation}\label{5eq:oscillation-lambda}
			\liminf \limits_{t\to 0^+} \lambda(t) < \limsup\limits_{t\to 0^+} \lambda(t)\,.
		\end{equation}
	\end{theorem}
	
	We will prove the theorem below.
	
	\begin{remark}[Original system]
		The construction made in this Section depends only on local properties of parabolic equations. 
		Therefore, we think that a completely analogous construction can be made to obtain an axisymmetric solution to the original model on the sphere with an oscillatory behavior, that means an initial datum $u_0$ on $(-1,1)$ such that the solution $u$ of \eqref{3eq:ob2}-\eqref{3eq:ob5} satisfies \eqref{5eq:oscillation} and \eqref{5eq:oscillation-lambda}.
	\end{remark}
	
	In the remainder of this section we prove Theorem \ref{5eq:oscillation}.
	
	\begin{lemma}[Rough estimates]
		\label{5lem:rough}
		Let $u$ denote the solution of \eqref{5eq:ob1}-\eqref{5eq:ob5}, \eqref{5eq:ob4a} with $I=(-5,5)$.
		There exists $T>0$ such that for all $0\leq t\leq T$
		\begin{align}
			u(\cdot,t) &>0 \quad\text{ in } \big[-\frac{7}{2},-\frac{3}{2}\big],
			\label{5eq:lemu-}\\
			\frac{11}{120}\leq \lambda(t) &\leq \frac{1}{2}\int_{-4}^{-1}g.
			\label{5eq:lemu-lambda}
		\end{align}
		Moreover, we have
		\begin{equation}
			1-\frac{g}{\lambda} \geq \frac{1}{11}\quad\text{ in }(-5,5)\setminus \big[-\frac{7}{2},-\frac{3}{2}\big].
			\label{5eq:lemu-rhs}
		\end{equation}
	\end{lemma}
	
	\begin{proof}
		We have $u_0\geq u_0^-$ and by comparison principle $u\geq u^-$ and $\lambda\leq \lambda^-$, where $(u^-,\lambda^-)$ denotes the solution of the obstacle problem \eqref{5eq:ob1}-\eqref{5eq:ob4}, \eqref{5eq:ob4a} with initial datum $u^-(\cdot,0)=u_0^-$ in $I=(-5,5)$.
		(For the comparison principle, compare \cite[Theorem 3.1 and (3.4)]{LNRV21}.)
		
		For $u^-$ the continuity result as $t \to 0$ from \cite{LNRV23} applies and yields that the support of $u^-$ and $\lambda^-$ are continuous, which in particular implies the existence of $T>0$ such that \eqref{5eq:lemu-} holds and such that
		\begin{align*}
			\lambda(t)\leq \lambda^-(t)\leq \frac{3}{2}\lambda_0^- = \frac{3}{2}\frac{1}{3}\int_{-4}^{-1}g,
		\end{align*}
		which yields the upper estimate for $\lambda(t)$ in \eqref{5eq:lemu-lambda}.
		
		The lower estimate follows from \eqref{5eq:lemu-}, since
		\begin{align*}
			\lambda(t) \geq \frac{1}{10}\int_{-\frac{7}{2}}^{-\frac{3}{2}}g
			\geq \frac{1}{10}\Big(\frac{6}{7}+\frac{1}{12}\Big)\geq \frac{11}{120}.
		\end{align*}
		Finally, we compute in $(-5,5)\setminus [-\frac{7}{2},-\frac{3}{2}]$
		\begin{equation*}
			1-\frac{g}{\lambda} =1-\frac{1}{12\lambda}\geq 1-\frac{10}{11}=\frac{1}{11}.
		\end{equation*}
	\end{proof}

	\begin{lemma}[Finer estimates]
		\label{5lem:finer}
		For all $\eps>0$ there exists $T(\eps)>0$ such that for all $0\leq t\leq T(\eps)$
		\begin{align}
			u(\cdot,t) &>0 \quad\text{ in }[-4+\eps,-1-\eps],
			\label{5eq:lem2:pos}\\
			u(\cdot,t) &=0 \quad\text{ in }(-5,5)\setminus \big([-4-\eps,-1+\eps]\cup [-\eps,\eta+\eps]\big).
			\label{5eq:lem2:vanish}
		\end{align}
		Moreover, for all $0<t<T(\eps)$
		\begin{equation}
			|\lambda(t)-\lambda_0^-| \leq C(\eps+\eta),
			\label{5eq:lambda:diff}
		\end{equation}
		where $\lambda_0^-=\frac{1}{3}\int\limits_{-4}^{-1}g$.
	\end{lemma}
	
	\begin{proof}
		The estimate \eqref{5eq:lem2:pos} follows as in Lemma \ref{5lem:rough} by comparison with $u^-$ and the continuity of the support of $u^-$.
		
		Next we let $(S(t))_{t>0}$ denote the heat semigroup associated to the Neumann problem on $(-5,5)$ and define
		\begin{equation*}
			v(\cdot,t) = S(t)u_0 + 10t.
		\end{equation*}
		We deduce from \eqref{5eq:lemu-lambda}
		\begin{equation*}
			\partial_t u - u'' = -1+ \frac{g}{\lambda} \leq -1 + \frac{6}{7}\cdot \frac{120}{11} 
			\leq 10 = \partial_t v -v''
		\end{equation*}
		and by the maximum principle $u\leq v$.
		By upper heat kernel bounds we deduce that $v\to 0$ uniformly away from the support of $u_0$, in particular
		\begin{equation*}
			u(x,t) \leq \omega_\eps(t) \quad\text{ for all }x\in (-5,5)\setminus \big([-4-\eps/2,-1+\eps/2]\cup [-\eps/2,\eta+\eps/2]\big),
		\end{equation*}
		where $\omega_\eps$ depends only on $\eps>0$ and $\int_{-5}^5 u_0$ and satisfies $\omega_\eps(t)\to 0$ as $t\downarrow 0$.
		
		We then can apply Lemma \ref{4lem:maxprinciple} and deduce \eqref{5eq:lem2:vanish}.
		
		The estimate \eqref{5eq:lambda:diff} follows from \eqref{5eq:lem2:pos} and \eqref{5eq:lem2:vanish}.
	\end{proof}
	
	We next define $u_r,u_\ell:\R\times [0,T]\to\R^+_0$ by
	\begin{equation*}
		u_r(x,t)=
		\begin{cases}
			u(x,t) \quad&\text{ if }-\frac{1}{2}<x<5,\\
			0 &\text{ else,}
		\end{cases}
	\end{equation*}
	and
	\begin{equation*}
		u_\ell(x,t)=
		\begin{cases}
			u(x,t) \quad&\text{ if }-5<x<-\frac{1}{2},\\
			0 &\text{ else.}
		\end{cases}
	\end{equation*}
	
	Then $u_r$ solves
	\begin{align}
		\partial_tu_r - u_r'' &= -fH(u_r)\quad\text{ in }\R,
		\label{5eq:tilde-u:1r}\\
		u_r(\cdot,0) &= u_0^+,
		\label{5eq:tilde-u:2r}
	\end{align}
	where $f(t)=1-\frac{g_{\min}}{\lambda(t)}$ and $u_\ell$ solves
	\begin{align}
		\partial_tu_\ell - u_\ell'' &= -f_\ell H(u_\ell)\quad\text{ in }\R,
		\label{5eq:tilde-u:1l}\\
		u_\ell(\cdot,0) &= u_0^-,
		\label{5eq:tilde-u:2l}
	\end{align}
	with $f_\ell(x,t)=1-\frac{g(x)}{\lambda(t)}$.
	
	By Lemma \ref{5lem:finer} we deduce that $f$ is nearly constant, more precisely 
	\begin{equation}
		|f(t)-f^+| \leq C(\eps+\eta) \quad\text{ for all }0<t<T(\eps)
		\label{5eq:ff-}
	\end{equation}
	with $f^+:= 1-\frac{g_{\min}}{\lambda_0^-}$.
	
	\begin{proof}[Proof of Theorem \ref{5thm:main}]
		So far we have chosen $u_0^-$.
		We will now define $u_0^+$ and thus also $u_0$. 
		Consider $\lambda_0:=f^+$ in Theorem \ref{4thm:main} and let $\kappa>0$ be chosen as in that theorem.
		By \eqref{5eq:ff-} we can next fix $\eta>0$, $\eps>0$ such that $|f(t)-\lambda_0|<\kappa$ and an initial datum $u_0^+$ as provided by Theorem \ref{4thm:main}.
		
		Now consider $\lambda=\lambda_0+\kappa$, $\mu=\lambda_0-\kappa$ and the solutions $u_\lambda$, $u_\kappa$ of \eqref{4eq:obst1-lambda}, \eqref{5eq:obst3-lambda}.
		By comparison principle the solution $u_r$ of \eqref{5eq:tilde-u:1r}, \eqref{5eq:tilde-u:2r} satisfies $u_\mu\leq u_r\leq u_\lambda$.
		
		By Theorem \ref{4thm:main} we conclude
		\begin{align*}
			\liminf \limits_{t\to 0^+} |\{u_r(\cdot,t)>0\}|
			&\leq \liminf \limits_{t\to 0^+} |\{u_\lambda(\cdot,t)>0\}|\\
			&< \limsup\limits_{t\to 0^+} |\{u_\mu(\cdot,t)>0\}|
			\leq \limsup\limits_{t\to 0^+} |\{u_r(\cdot,t)>0\}|.
		\end{align*}
		
		On the other hand we deduce from \eqref{5eq:lemu-rhs} that $f_\ell\leq -\frac{1}{11}$ in $\R\setminus \big[-\frac{7}{2},-\frac{3}{2}\big]$ and hence in a neighborhood of $\partial\{u_0^->0\}$.
		By the results in \cite{BF76} this is sufficient to conclude that 
		\begin{equation*}
			\lim\limits_{t\to 0^+} |\{u_\ell(\cdot,t)>0\}|
			= |\{u_0^- >0\}|=3.
		\end{equation*}
		Since
		\begin{equation*}
			u(\cdot,t)=
			\begin{cases}
				u_\ell(\cdot,t) \quad&\text{ on }\big(-5,-\frac{1}{2}\big)\,,\\
				u_r(\cdot,t)  \quad&\text{ on }\big(\frac{1}{2},5\big)\,,
			\end{cases}
		\end{equation*}
		we deduce \eqref{5eq:oscillation}.
		
		Finally, we have with $A_\ell(t)=\{u_\ell(\cdot,t)>0\}$ and $A_r(t)=\{u_\ell(\cdot,t)>0\}$
		\begin{align*}
			\lambda(t) &= \frac{1}{\big|A_\ell(t) \cup A_r(t)\big|}\Big(\int_{A_\ell(t)} g +  \int_{A_r(t)}g\Big)\\
			&= g_{\min} + \frac{1}{|A_\ell(t)|+|A_r(t)|}\Big(\int_{A_\ell(t)}(g-g_{\min})+(g_{\max}-g_{\min})|A_r(t)|\Big)
		\end{align*}
		We have $A_\ell(t)\to A_\ell(0)$ and
		\begin{equation*}
			a_r^- := \liminf_{t\downarrow 0}|A_r(t)|\,<\, \limsup_{t\downarrow 0}|A_r(t)|=:a_r^+\,.
		\end{equation*}
		This gives
		\begin{align*}
			&\limsup_{t\downarrow 0}\lambda(t)-\liminf_{t\downarrow 0}\lambda(t)\\
			&\qquad \geq \frac{1}{|A_\ell(0)|+a_r^+}\Big(\int_{A_\ell(0)}(g-g_{\min})+(g_{\max}-g_{\min})a_r^+\Big)\\
			&\qquad\qquad -\frac{1}{|A_\ell(0)|+a_r^-}\Big(\int_{A_\ell(0)}(g-g_{\min})+(g_{\max}-g_{\min})a_r^-\Big)\\
			&\qquad=\frac{\int_{A_\ell(0)}(g_{\max}-g)}{(|A_\ell(0)|+a_r^+)(|A_\ell(0)|+a_r^-)}(a_r^+-a_r^-)\,>\,0\,.
		\end{align*}
	\end{proof}
	
	\bigskip 
	{\bf Acknowledgments:} The authors gratefully acknowledge the financial support of  the Bonn International Graduate School of Mathematics at the Hausdorff Center for Mathematics (EXC 2047/1, Project-ID 390685813) funded through the Deutsche Forschungsgemeinschaft (DFG, German Research Foundation).

	\bibliographystyle{plainurl}
	\bibliography{\jobname}
	
\end{document}